\documentclass[12pt]{amsart}
\usepackage{amsmath}
\usepackage{amscd}
\usepackage{amssymb}
\usepackage[dvips]{graphicx}

\title[Hamiltonain vector fields]{Hamiltonian vector fields of homogeneous polynomials in two variables}
\author{Sergiy Maksymenko}

\newtheorem{thm}[subsection]{Theorem}

\newtheorem{lem}[subsection]{Lemma}
\newtheorem{cor}[subsection]{Corollary}

\newtheorem{prop}[subsection]{Proposition}

\newtheorem{rem}[subsection]{Remark}
\newtheorem{exmp}[subsection]{Example}
\newtheorem{defn}[subsection]{Definition}

\makeatletter
\@addtoreset{subsection}{section}
\@addtoreset{equation}{section}
\@addtoreset{figure}{section}
\@addtoreset{table}{section}
\makeatother

\providecommand\eqref[1]{(\ref{#1})}

\newcommand\CCC{{\mathbb C}}
\newcommand\NNN{{\mathbb N}}
\newcommand\RRR{{\mathbb R}}

\newcommand\ZZZ{{\mathbb Z}}

\newcommand\XX{{\mathcal X}}
\newcommand\YY{{\mathcal Y}}

\newcommand\id{\mathrm{id}}
\newcommand\IM{\mathrm{Im}}
\newcommand\RE{\mathrm{Re}}

\newcommand\supp{\mathrm{supp\,}}

\newcommand\Diff{\mathcal{D}}

\newcommand\End{{\mathcal{E}}}
\newcommand\EndId{\End_{\id}}

\newcommand\DiffB{\Diff(\BFld)}

\newcommand\manif{M}

\newcommand\AFlow{\mathbf{F}}
\newcommand\AFld{F}
\newcommand\BFlow{\mathbf{G}}
\newcommand\BFld{G}
\newcommand\HFld{H}

\newcommand\Dom[1]{\mathcal{U}_{#1}}

\newcommand\BDom{\Dom{\BFlow}}

\newcommand\secfunc{\psi}

\newcommand\lShift{\sigma}
\newcommand\iShift{\secfunc}
\newcommand\ShiftI{\Psi}

\newcommand\szShift{\Lambda}  

\newcommand\BFunc{g}

\newcommand\Amap{A}
\newcommand\Bmap{B}
\newcommand\Pmap{P}
\newcommand\Qmap{Q}
\newcommand\Rmap{R}

\newcommand\difM{h}

\newcommand\qdifM{q}

\newcommand\afunc{\alpha}
\newcommand\bfunc{\beta}

\newcommand\dg{p}

\newcommand\nbh{V}
\newcommand\anbh{U}
\newcommand\Nbh{\mathcal{U}}

\newcommand\EndFlowNbh[2]{\End(#1,#2)}
\newcommand\EndFlow[1]{\End(#1)}
\newcommand\EndIdFlowNbh[2]{\EndId(#1,#2)}
\newcommand\EndIdFlow[1]{\EndId(#1)}

\newcommand\EndBV{\EndFlowNbh{\BFld}{\nbh}}

\newcommand\EndAU{\EndFlowNbh{\AFld}{\anbh}}
\newcommand\EndBVi[1]{\End_{#1}(\BFld,\nbh,\orig)}

\newcommand\EndAUDHi[1]{\End_{#1}(\AFld,\anbh,\dHsp)}
\newcommand\EndZAUDHi[1]{\End_{#1}(\AFld,\anbh,\dHsp)_{\ZZZ}}

\newcommand\EndB{\EndFlow{\BFld}}
\newcommand\EndIdBV{\EndIdFlowNbh{\BFld}{\nbh}}
\newcommand\EndIdB{\EndIdFlow{\BFld}}

\newcommand\gEnd{\hat{\End}}
\newcommand\gEndBz{\gEnd(\BFld,\orig)}

\newcommand\gEndIdBz{\hat{\End}_{\id}(\BFld,\orig)}

\newcommand\AST{{\rm($\ast$)}}

\newcommand{\Flat}{\mathrm{Flat}\,} 
\newcommand{\Func}{\mathrm{Func}\,} 
\newcommand{\FlatMap}{\mathrm{Map}\,} 

\newcommand\Hsp{\mathbb{H}}
\newcommand\rrho{\rho}

\newcommand\AZHR{\Flat_{\ZZZ}(\Hsp,\partial\Hsp)}
\newcommand\EZHH{\FlatMap^{\infty}_{\ZZZ}(\Hsp,\partial\Hsp)}
\newcommand\ARtwo{\Flat(\RRR^2,\orig)}
\newcommand\ERtwo{\FlatMap^{\infty}(\RRR^2,\orig)}

\newcommand\IHsp{\overset{\circ}{\Hsp}}
\newcommand\dHsp{\partial\Hsp}

\newcommand\EZHdH{\FlatMap^{0}_{\ZZZ}(\Hsp,\dHsp)}
\newcommand\ERtwoz{\FlatMap^{0}(\RRR^2,\orig)}

\newcommand\hdifM{\hat\difM}

\newcommand\EndBViz[1]{X_{\Qmap}} 

\newcommand\fbij{\mathbf{f}}
\newcommand\mbij{\mathbf{m}}

\newcommand\orig{O}

\newcommand\Sh{Sh}

\begin{document}
\begin{abstract}
Let $g:\mathbb{R}^{2}\to\mathbb{R}$ be a homogeneous polynomial of degree $\dg\geq2$, 
$G=(-g'_{y}, g'_{x})$ be its Hamiltonian vector field, and $\mathbf{G}_{t}$ be the local flow generated by $G$.
Denote by ${\mathcal{E}}(G,O)$ the space of germs of $C^{\infty}$ diffeomorphisms $(\mathbb{R}^{2},O)\to (\mathbb{R}^{2},O),$ that preserve orbits of $\BFld$.
Let also $\hat{\mathcal{E}}_{\mathrm{id}}(G,O)$ be the identity component of $\hat{\mathcal{E}}(G,O)$ with respect to $C^1$ topology.

Suppose that $g$ has no multiple prime factors.
Then we prove that for every $h\in\hat{\mathcal{E}}_{\mathrm{id}}(G,O)$ there exists a germ of a smooth function $\alpha:\mathbb{R}^{2}\to\RRR$ at $O$ such that 
$$ 
h(z)=\mathbf{G}_{\alpha(z)}(z).
$$ 
\end{abstract}
\maketitle
	
\section{Introduction}
Let $\dg\geq1$ and $\BFunc:\RRR^2\to\RRR$ be a homogeneous polynomial of degree $\dg+1$, i.e.\! $\deg\BFunc\geq2$.
Then we have a prime decomposition of $\BFunc$ over $\RRR$:
\begin{equation}\label{equ:g-decomp}
\BFunc(x,y) =  \prod_{i=1}^{l} L_i(x,y) \cdot \prod_{j=1}^{\dg+1-l} Q_j(x,y),
\end{equation}
where every $L_i = a_i x + b_i y$ is a linear function, and every $Q_j$ is a definite quadratic form.

\begin{lem}\label{lm:AST-for-polynom}{\rm\cite{Maks:TaylorSer}}
The following conditions for a homogeneous polynomial $\BFunc$ of degree $\deg\BFunc\geq2$ are equivalent:
\begin{enumerate}
 \item
decomposition~\eqref{equ:g-decomp} contains no multiple factors
 \item
none of the partial derivatives $\BFunc'_x$ and $\BFunc'_{y}$ is identically zero (i.e. $\BFunc$ does depend on $x$ and $y$) and these polynomials are relatively simple in the ring $\RRR[x,y]$.
\end{enumerate}
In this case the origin $\orig\in\RRR^2$ is a unique critical point for $\BFunc$.
\end{lem}

\begin{defn}[Property \AST\ for a polynomial]
Say that a homogeneous polynomial $\BFunc \in \RRR[x,y]$ of degree $\deg\BFunc\geq2$ has property \AST\ if it satisfies one of the conditions of Lemma~\ref{lm:AST-for-polynom}.
\end{defn}

\begin{exmp}
For $n\geq2$ consider the following function 
$$
\omega_n:\CCC\to\CCC, \qquad 
\omega_n(z)=z^n.
$$
Then its real and imagine parts $\RE(z^n)$ and $\IM(z^n)$ have property \AST.
\end{exmp}

Let $\HFld=(-\BFunc'_{y},\BFunc'_{x})$ be the Hamiltonian vector field for $\BFunc$.
Then $\BFunc$ is constant along orbits of $\HFld$.
The typical foliations of $\RRR^2$ by level sets of homogeneous polynomials are shown in Figures~\ref{fig:non-def} and~\ref{fig:def}.

Notice that the property \AST\ for $\BFunc$ can be formulated as follows: 
\emph{the Hamiltonian vector field $\HFld$ of $\BFunc$ can not be represented as a product $\HFld = \omega \HFld_1$, where $\omega$ is a homogeneous polynomial of degree $\deg\omega\geq1$ and $\HFld_1$ is a homogeneous vector field.}

\begin{defn}[Property \AST\ for a vector field]
Say that a vector field $\BFld$ on $\RRR^2$ {\bfseries has property \AST\ at $\orig$} if there exist
a smooth ($C^{\infty}$) and everywhere non-zero function $\eta:\RRR^2\to\RRR\setminus\{0\}$, local coordinates $(x,y)$ at $\orig$, and a homogeneous polynomial $\BFunc(x,y)$ having property \AST\ such that 
$$\BFld=\eta\HFld,$$
where $\HFld=(-\BFunc_y,\BFunc_x)$ is a Hamiltonian vector field of $\BFunc$.
\end{defn}

It follows from Lemma~\ref{lm:AST-for-polynom} that in this case the origin $\orig\in\RRR^2$ is an isolated singular point of $\BFld$.

\subsection{Main result.}
Let $\BFld$ be a smooth vector field defined in a neighborhood of the origin $\orig\in\RRR^2$.
Denote by $\gEndBz$ the set of germs of $C^{\infty}$ diffeomorphisms 
$$\difM:(\RRR^2,\orig)\to (\RRR^2,\orig)$$
preserving orbits of $\BFld$, i.e.\! $\difM\in\gEndBz$ if there exists a neighborhood $\nbh$ of $\orig$ such that 
\begin{equation}\label{equ:h_o__o}
\difM(\omega\cap\nbh)\subset\omega
\end{equation}
for each orbit $\omega$ of $\BFld$.

Let also $\gEndIdBz$ be the \emph{identity component\/} of $\gEndBz$ with respect to $C^1$-topology.
It consists of germs of diffeomorphisms at $\orig$ isotopic to $\id_{\RRR^2}$ in $\gEndBz$ via isotopy  whose partial derivatives of the first order continuously depend on the parameter, see~\cite{Maks:TaylorSer} for details.

Denote by $\BFlow:\; \RRR^2\times\RRR \supset \BDom\;\longrightarrow\;\RRR^2$ the corresponding local flow of $\BFld$ defined on an open neighborhood $\BDom$ of $\RRR^2\times\{0\}$ in $\RRR^2\times\RRR$.

Then for every germ of a smooth function $\afunc:\RRR^2\to\RRR$ at $\orig$ we can define the following map $\difM:\RRR^2\to\RRR^2$ by 
\begin{equation}\label{equ:sm-shift-def}
\difM(z)=\BFlow(z,\afunc(z)).
\end{equation}
This map will be called the \emph{smooth shift} along orbits of $\BFld$ via the function $\afunc$.
Denote by $\Sh(\BFld,\orig)$ the set of germs of mappings of the form~\eqref{equ:sm-shift-def}, where $\afunc$ runs over all germs of smooth function at $\orig$.

Then, see~\cite{Maks:TA:2003}, $\Sh(\BFld,\orig) \subset\gEndIdBz$.

In this paper we prove the following theorem:
\begin{thm}\label{th:shift-func-for-inf-close}
Let $\BFld$ be a vector field on $\RRR^2$ having property \AST\ at $\orig$.
Then $\Sh(\BFld,\orig) =\gEndIdBz$.
Thus every $\difM\in\gEndIdBz$ can be represented in the form~\eqref{equ:sm-shift-def} for some smooth function $\afunc:\RRR^2\to\RRR$.
\end{thm}

\begin{rem}\rm
Suppose that $\orig$ is a regular point for $\BFld$, i.e. $\BFld(\orig)\not=0$.
Then every smooth map preserving orbits of $\BFld$ is a neighborhood of $\orig$ is a shift along orbits of $\BFld$ via a certain \emph{smooth\/} function $\afunc$.
For the convenience of the reader we recall a proof of this fact, see~\cite[Eq.~(10)]{Maks:TA:2003}.
Indeed, since $\BFld(\orig)\not=0$, it follows that there are local coordinates $(x_1,\ldots,x_n)$ at $\orig$ such that $\BFld(x)=(1,0,\ldots,0)$, whence 
$$\BFlow(x_1,\ldots,x_n,t) = (x_1+t,x_2,\ldots,x_n).$$
If now $\difM=(\difM_1,\ldots,\difM_n):\RRR^n\to\RRR^n$ is a smooth map that preserves orbits of $\BFld$, then $\difM_i=x_i$ for $2\leq i \leq n$.
Set 
\begin{equation}\label{equ:reg-shifts}
\afunc(x) = \difM_1(x)-x_1.
\end{equation}
Then $\difM(x) = \BFlow(x,\afunc(x))$.
\end{rem}

\subsection{Applications.}
In~\cite{Maks:TA:2003} the identity $$\Sh(\BFld,\orig) = \gEndIdBz$$ is established for all linear  vector fields on $\RRR^n$.
Thus if $\BFld(x)=A\cdot x$ is a linear vector field on $\RRR^n$, where $A$ is a non-zero $(n\times n)$-matrix, then every $\difM\in\gEndIdBz$ can be represented as follows
$$\difM(x) = e^{\afunc(x) A} \cdot x$$
for a certain smooth function $\afunc:\RRR^n\to\RRR$.
It allowed for a vector field $\BFld$ satisfying mild conditions describe the homotopy types of the connected components of the group $\DiffB$ of orbit preserving diffeomorphisms.
This result was essentially used in~\cite{Maks:AGAG:2006} for the calculation of the homotopy types of stabilizers and orbits of Morse functions on compact surfaces $M$ with respect to the action of $\Diff(M)$.

Theorem~\ref{th:shift-func-for-inf-close} allowed to perform similar calculation for large class of functions on surfaces with isolated singularities.
This will be done in another paper.

\subsection{Structure of the paper.}
In Section~\ref{sect:cont-maps} the definition of weak Whitney topologies is given.

Section~\ref{sect:proof-th:shift-func-for-inf-close} includes a plan of the proof of Theorem~\ref{th:shift-func-for-inf-close}.
Using results of~\cite{Maks:TaylorSer} the proof is reduced to the case when $\difM\in\gEndIdBz$ is $\infty$-close to he identity at  $\orig$, see Proposition~\ref{pr:inf-id-shift-func}.
It turns out that in order to work with these mappings it is convenient to use polar coordinates $(\phi,\rrho)$, see Section~\ref{sect:polar-coordinates}.
In this case instead of a unique singular point $\orig=(0,0)\in\RRR^2$ we obtain a whole line of singular points $\rrho=0$, but the formulas for the vector field $\BFld$ in polar coordinates becomes essentially simple.

Then in Section~\ref{sect:corresp-flat-func} it is shown that instead of smooth functions on $\RRR^2$ that are flat at $\orig$, we can consider smooth functions with respect to polar coordinates $(\phi,\rrho)$ being flat for $\rrho=0$.
Similarly, in Section~\ref{sect:corresp-flat-maps} it is proved that instead of diffeomorphisms of $\RRR^2$ that are $\infty$-close to the identity at $\orig$ it is possible to consider diffeomorphisms of the half-plane of polar coordinates $\Hsp$ that are $\infty$-close to the identity for $\rrho=0$.

In Section~\ref{sect:proof_of_shift2} a proof of Proposition~\ref{pr:inf-id-shift-func} is given.
This will complete Theorem~\ref{th:shift-func-for-inf-close}.

\section{Continuous maps between functional spaces}\label{sect:cont-maps}
Let $\nbh\subset\RRR^n$ be an open subset and $f=(f_1,\ldots,f_m):\nbh\to\RRR^m$ be a smooth mapping.
For every compact $K\subset \nbh$ and integer $r\geq 0$ define the \emph{$r$-norm} of $f$ on $K$ by
$$
\|f\|^{r}_{K} \;= \;\sum_{j=1}^{m} \;\sum_{|i|\leq r}\; \sup\limits_{x\in K} |D^{i}f_j(x)|,
$$
where $i=(i_1,\ldots,i_n)$, $|i|=i_1+\cdots+i_n$, and
$D^i = \frac{\partial^{|i|}}{\partial x_1^{i_1}\cdots\partial x_{n}^{i_n}}$.
For a fixed $r$ the norms $\|f\|^{r}_{K}$ define the so-called \emph{weak} $C^{r}_{W}$ Whitney topology on $C^{\infty}(\nbh,\RRR^m)$, see~\cite{GolubitskyGuillemin,Hirsch:DiffTop}.

\begin{defn}
Let $A,B,C,D$ be smooth manifolds, 
$$\XX\subset C^{\infty}(A,B), \qquad \YY\subset C^{\infty}(C,D)$$
be two subsets and $F:\XX\to\YY$ be a map.
Say that $F$ is {\bfseries $C^{s,r}_{W,W}$-continuous} provided it is continuous from $C^{s}_{W}$-topology on $\XX$ to $C^{r}_{W}$-topology on $\YY$.

Say that $F$ is {\bfseries tamely continuous} if for every $r\geq 0$ there exists an integer number $s(r)\geq0$ such that $F$ is $C^{s(r),r}_{W,W}$-continuous.
Evidently, every tamely continuous map is $C^{\infty,\infty}_{W,W}$-continuous.
\end{defn}

The following lemmas are easy to prove, see~\cite{Maks:TaylorSer}.
\begin{lem}\label{lm:D-cont}
Let $D:C^{\infty}(\nbh)\to C^{\infty}(\nbh)$ be the mapping defined by
$$
D(\afunc) = \frac{\partial^{|k|}\afunc}{\partial x^k},
$$
where $k=(k_1,\ldots,k_n)$, $|k|=\sum\limits_{i=1}^{n} k_i$, and $\partial x^k=\partial x_1^{k_1}\cdots\partial x_n^{k_n}$.
Then $D$ is $C^{r+|k|,r}_{W,W}$-continuous for all $r\geq0$.
\end{lem}

\begin{lem}\label{lm:Z-cont}
Let $Z:C^{\infty}(\nbh)\to C^{\infty}(\nbh)$ be the mapping defined by
$$
Z(\afunc)(x_1,\ldots,x_n) = x_1 \cdot \afunc(x_1,\ldots,x_n), \qquad \afunc\in C^{\infty}(\nbh).
$$
Then $Z$ is injective and for every $r\geq0$ the mapping $Z$ is $C^{r,r}_{W,W}$-continuous and its inverse $Z^{-1}$ is a $C^{r+1,r}_{W,W}$-continuous.
\end{lem}

\begin{lem}[Hadamard]
Let $f:\RRR\to\RRR$ be a smooth function such that $f(0)=0$, then 
$f(x)=  s\underbrace{ \int_{0}^{1} f'(tx)dt}_{g(x)} = x\, g(x)$, where $g$ is smooth and $g(0)=f'(0)$.\qed
\end{lem}

More generally,
\begin{equation}\label{equ:Hadamard-fxy}
f(x+y) = f(x) + y \underbrace{ \int_{0}^{1} f'(x+sy)ds }_{g(x,y)} = f(x) + y \cdot g(x,y),
\end{equation}
where $g$ is also smooth and such that $g(x,0)=f'(x)$.

In particular, if $f$ has an inverse function $f^{-1}$ then
\begin{equation}\label{equ:Hadamard-ffinvxy}
f(f^{-1}(x)+y) =
f(f^{-1}(x)) + y \cdot g(f^{-1}(x),y) =
x + y \cdot g(f^{-1}(x),y).
\end{equation}

\section{Proof of Theorem~\ref{th:shift-func-for-inf-close}}
\label{sect:proof-th:shift-func-for-inf-close}
Actually we establish a more general statement. 
First we introduce some notation.

\subsection{Smooth shifts along orbits of vector fields.}\label{sect:smooth-shifts}
Let $\BFld$ be a vector field on a manifold $\manif$.
We will always denote by $$\BFlow:\manif\times\RRR\supset\BDom \to M$$ the local flow of $\BFld$, where $\BDom$ is an open neighborhood of $\manif\times0$ in $\manif\times\RRR$.

For every open subset $\nbh \subset\manif$ let
$$\EndBV \; \subset\; C^{\infty}(\nbh,\manif)$$
be the set of all smooth mappings $\difM:\nbh\to\manif$ such that
\begin{enumerate}
\item[(1)]
$\difM(\omega\cap\nbh)\subset\omega$ for every orbit $\omega$ of $\BFld$, in particular $\difM$ is fixed on the set of singular points of $\BFld$ contained in $\nbh$;
\item[(2)]
$\difM$ is a local diffeomorphism at every singular point of $\BFld$.
\end{enumerate}

Let also $\EndIdBV$ be the subset of $\EndBV$ consisting of mappings $\difM$ such that
\begin{enumerate}
\item[(3)]
$\difM$ is homotopic to $\id_{\manif}$ in $\EndBV$.
\end{enumerate}

If $\nbh=\manif$, then $\End(\BFld,\manif)$ and $\End_{\id}(\BFld,\manif)$ will be denoted by $\EndB$ and $\EndIdB$ respectively.

Let $\orig\in\nbh$ be a singular point of $\BFld$.
Then $\difM(\orig)=\orig$ for every $\difM\in\EndBV$.
Denote by $\EndBVi{\infty}$ the subset of $\EndBV$ consisting of mappings $\difM$ which are \emph{$\infty$-close to the identity at $\orig$}, i.e.\! the $\infty$-jets of $\difM$ and $\id_{\nbh}$ at $\orig$ coincide.

\begin{thm}\label{th:global-shift-func-1}
Let $\BFld$ be a vector field on $\RRR^2$ having property \AST\ at $\orig$ and $\nbh$ be a sufficiently small open neighborhood of $\orig$.
Then for every $f\in\EndIdBV$ there exists a neighborhood $\Nbh_{f}$ in $\EndIdB$ with respect to $C^{\dg}_{W}$-topology and a tamely continuous map
$$
\lShift_{\nbh}\;:\;\EndIdBV \supset \Nbh_{f} \;\longrightarrow\; C^{\infty}(\nbh)
$$
such that 
$$
\difM(z) = \BFlow(z, \lShift_{\nbh}(\difM)(z))
$$
for every $\difM\in\Nbh_{f}$.

Moreover, if $\deg\BFunc\geq3$, then $\lShift$ can be defined on all of $\EndIdBV$.
\end{thm}
The proof is based on the following two statements.
The first one is established in~\cite{Maks:TaylorSer}:
\begin{prop}\label{pr:almost-shift-func}{\rm\cite{Maks:TaylorSer}}
Let $\BFld$ be a vector field on $\RRR^2$ having property \AST\ at $\orig$ and $\anbh\subset\nbh$ be two sufficiently small open neighborhoods of $\orig$.
Then for every $f\in\EndIdBV$ there exists a neighborhood $\Nbh_{f}$ in $\EndIdBV$ with respect to $C^{\dg}_{W}$-topology and a tamely continuous map
$$
\szShift:\Nbh_{f} \to C^{\infty}(\nbh)
$$
such that for every $\difM\in\Nbh_{f}$ we have that $$\supp\szShift(\difM) \subset \anbh$$ and the mapping $\hdifM:\nbh\to\RRR$ defined by
$$
\hdifM(z) = \BFlow(\difM(z), -\szShift(\difM)(z))
$$
is $\infty$-close to $\id_{\RRR^2}$ at $\orig$.
In particular, $\hdifM\in\EndBVi{\infty}$.

Moreover, if $\deg\BFunc\geq3$, then $\szShift$ can be defined on all of $\EndIdB$.
\end{prop}

The second statement will be proved in Section~\ref{sect:proof_of_shift2}.
\begin{prop}\label{pr:inf-id-shift-func}
Let $\BFld$ be a vector field on $\RRR^2$ having property \AST\ at $\orig$ and $\nbh$ be a sufficiently small open neighborhood of $\orig$.
Then there exists a unique map
$$
\ShiftI:\EndBVi{\infty} \to \Flat(\RRR^2,\orig)
$$
such that for every $\hdifM\in\EndBVi{\infty}$ we have that
\begin{equation}\label{equ:shfunc-for-hh}
\hdifM(z) = \BFlow(z,\ShiftI(\hdifM)(z))
\end{equation}
This mapping is $C^{3r+\dg,r}_{W,W}$-continuous for every $r\geq0$.
\end{prop}

Now we can complete Theorem~\ref{th:global-shift-func-1}.
First notice that for a smooth function $\afunc$ and a mapping $\difM$ the following relations are equivalent:
\begin{equation}\label{equ:equiv-shifts}
\difM(z) = \BFlow(z,\afunc(z)) \qquad\text{and}\qquad  z = \BFlow(\difM(z),-\afunc(z)).
\end{equation}

Let $f\in\EndIdB$.
Then it follows from Proposition~\ref{pr:almost-shift-func} that for every $f\in\EndIdB$ there exists a $C^{\dg}_{W}$-neighborhood $\Nbh_{f}$ of $f$ in $\EndIdB$ and a well-defined map
$$
H:\Nbh_{f} \to \EndBVi{\infty}
$$
given by
$$
H(\difM)(z) = \BFlow(\difM(z), - \szShift(\difM)(z)).
$$

Then the following map $\lShift:\Nbh_{f} \to C^{\infty}(\nbh)$ defined by
$$\lShift  = \szShift + \ShiftI \circ H$$
satisfies the statement of our theorem.

Indeed, let $\difM\in\Nbh_{f}$ and $\hdifM=H(\difM)$.
Then
$$
\lShift(\difM) = 
\szShift(\difM) + \ShiftI \circ H(\difM) = 
\szShift(\difM) + \ShiftI(\hdifM).
$$

Whence
$$
\begin{array}{rcl}
 \BFlow\bigl(\difM(z), -\lShift(\difM)(z) \bigr) & = & 
\BFlow\bigl(\difM(z), -\szShift(\difM)(z) - \ShiftI(\hdifM)(z)\bigr)  = \\ [2mm]
& = & \BFlow\bigl( 
        \;\underbrace{ \BFlow\bigl(\difM(z), -\szShift(\difM)(z)\bigr) }_{\hdifM}\,, -\ShiftI(\hdifM)(z)\;
             \bigr)   = \\ [7mm]
& = & \BFlow\bigl( \hdifM(z), -\ShiftI(\hdifM)(z)\bigr) 
 \overset{\eqref{equ:shfunc-for-hh}~\text{òà}~\eqref{equ:equiv-shifts}}{=\!=\!=\!=\!=\!=\!=}  z,
\end{array}
$$
Therefore 
$$ \difM(z) = \BFlow(z, \lShift(\difM)(z)).$$
If $\deg\BFunc\geq3$, then $\lShift$ is defined on all of $\EndIdB$.

Theorem~\ref{th:global-shift-func-1} is completed modulo Proposition~\ref{pr:inf-id-shift-func}.
The proof of this proposition will be given in Section~\ref{sect:proof_of_shift2}.

%
%

\section{Polar coordinates}\label{sect:polar-coordinates}
Let $\Hsp = \{ (\phi,\rrho) \ | \ \rrho\geq 0\}  \subset \RRR^2$ be the closed upper half-plane of $\RRR^2$ with cartesian coordinates which we denote by $(\phi,\rrho)$.
Let also $\dHsp=\{\rrho=0\}$ be its boundary (i.e. $\phi$-axis), and $\IHsp=\{\rrho>0\}$ the interior of $\Hsp$.
Take another copy of $\RRR^2$ with coordinates $(x,y)$ and consider the following map 
$$\Pmap_k:\Hsp \to \RRR^2,\qquad
\Pmap_k(\phi,\rrho) = (\rrho^k \cos\phi, \rrho^k \sin\phi).
$$
For $k=1$ this map defines the so-called \emph{polar\/} coordinates in $\RRR^2$.
We will also denote the mapping $\Pmap_1$ simply by $\Pmap$.

Evidently, $\Pmap_k(\dHsp)=0\in \RRR^2$ and the restriction of $\Pmap_k$ onto $\IHsp$ is a $\ZZZ$-covering map:
$\Pmap_k:\IHsp \to \RRR^2\setminus\{\orig\}$, where the group $\ZZZ$ acts on $\Hsp$ by $n\cdot(\phi,\rrho) = (\phi+2\pi n,\rrho)$.

\begin{lem}\label{lm:pres-homog-pol}
Let $\BFunc:\RRR^2\to\RRR$ be a homogeneous polynomial of degree $\dg+1$ and $\phi_0\in\RRR$.
Then there are unique but not necessarily distinct numbers $\phi_i$, $(i=1,\ldots,l)$ such that
$$\phi_0-\frac{\pi}{2} \;\leq\; \phi_1 \;\leq \;\ldots\; \leq \phi_l \;<\; \phi_0+\frac{\pi}{2} $$
and a smooth function $\gamma$ such that $\gamma(\phi)\not=0$ for all $\phi\in(\phi_0-\frac{\pi}{2},\phi_0+\frac{\pi}{2})$ and
$$
\BFunc(\rrho\cos\phi,\rrho\sin\phi) = \rrho^{\dg+1}\cdot \gamma(\phi) \cdot \prod_{i=1}^{l} (\phi-\phi_i).
$$
\end{lem}
\begin{proof}
Notice that there exists a unique decomposition of $\BFunc$:
\begin{equation}\label{equ:Q-decomp}
\BFunc(x,y)  =   \tau(x,y) \cdot \prod_{i=1}^{l} (b_i x + a_i y),
\end{equation}
where 
$$
a_i=\cos\phi_i, \qquad b_i=\sin\phi_i,
$$
for a unique $\phi_i\in[\phi_0-\frac{\pi}{2}, \phi_0+\frac{\pi}{2})$, $(i=1,\ldots,l)$, such that $\phi_i\leq \phi_{i+1}$, and $\tau$ is a homogeneous polynomial of degree $\dg+1-l$ such that
$$\tau(x,y)\not=0, \qquad \text{for}\ (x,y)\not=0.$$
Therefore
$$
b_i x + a_i y  = \sin\phi_i \cdot \rrho\cos\phi + \cos\phi_i\cdot \rrho \sin\phi  = \rrho\cdot\sin(\phi-\phi_i),
$$
and thus
$$
\BFunc(\rrho\cos\phi,\rrho\sin\phi) = \rrho^{\dg+1}\cdot  \tau(\cos\phi,\sin\phi)\cdot  \prod_{i=1}^{l} \sin(\phi-\phi_i).
$$
Notice that the function $\frac{\sin(\phi-\phi_i)}{(\phi-\phi_i)}$ is smooth and strictly positive on the interval $(\phi_i-\pi,\phi_i+\pi)$ and $\tau(\cos\phi,\sin\phi)>0$ for every $\phi$, we obtain that
$$
\BFunc(\rrho\cos\phi,\rrho\sin\phi) = \rrho^{\dg+1}\cdot \gamma(\phi) \cdot \prod_{i=1}^{l} (\phi-\phi_i),
$$
for a certain smooth function $\gamma:\RRR\to\RRR$ such that $\gamma(\phi)\not=0$ for all $\phi\in(\phi_0-\frac{\pi}{2},\phi_0+\frac{\pi}{2})$.
\end{proof}

The level curves of a homogeneous polynomial $\BFunc:\RRR^2\to\RRR$ and the mapping $\BFunc\circ\Pmap_k:\Hsp\to\RRR$ are shown in Figure~\ref{fig:non-def} for $l=0$ and in Figure~\ref{fig:def} for $l\geq1$.

\begin{figure}[ht]
\begin{tabular}{ccc}
\includegraphics[height=3cm]{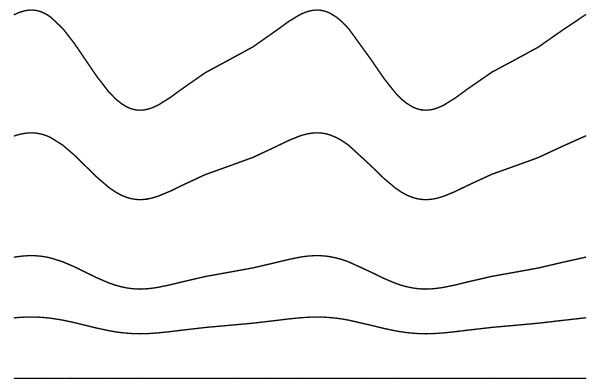} & \qquad \qquad &
\includegraphics[height=3cm]{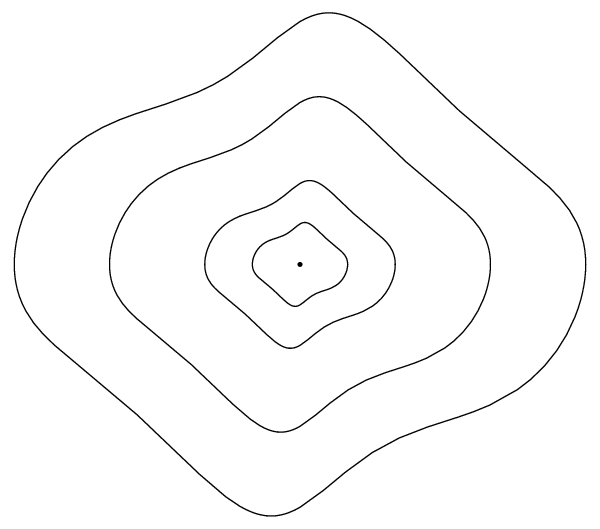} \\
$\Hsp$ & $\xrightarrow{~~~~~~\Pmap_k~~~~~~}$ & $\RRR^2$
\end{tabular}\caption{$l=0$.}
\protect\label{fig:non-def}
\end{figure}

\begin{figure}[ht]
\begin{tabular}{ccc}
\includegraphics[height=3cm]{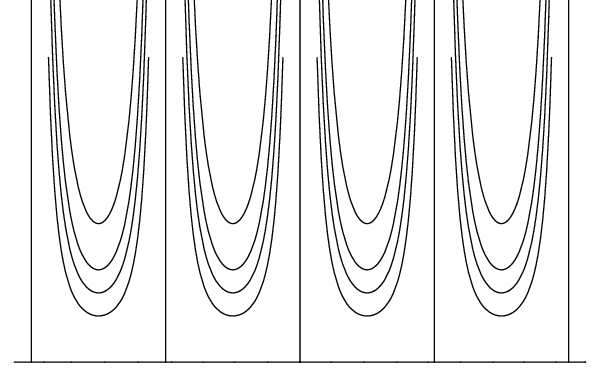} & \qquad \qquad &
\includegraphics[height=3cm]{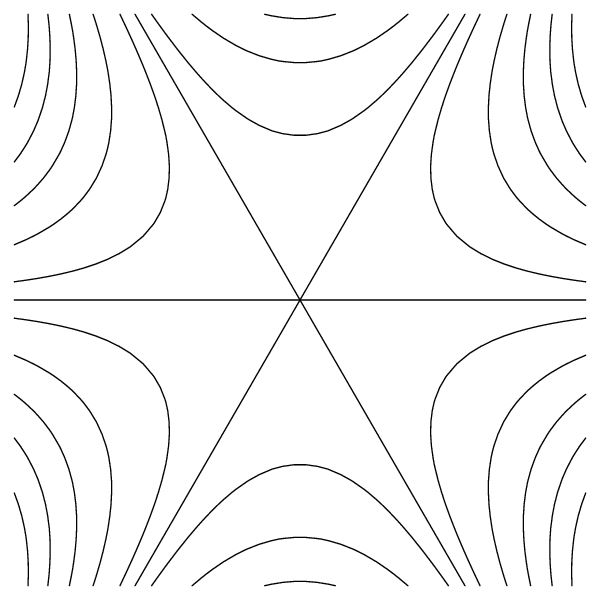} \\
$\Hsp$ & $\xrightarrow{~~~~~~\Pmap_k~~~~~~}$ & $\RRR^2$
\end{tabular}\caption{$l\geq1$.}
\protect\label{fig:def}
\end{figure}

\subsection{Lifting vector fields from $\RRR^2$ to $\Hsp$.}
Let $\BFld$ be a smooth vector field defined in a neighborhood $\nbh$ of $\orig\in\RRR^2$.
Denote $$\anbh=\Pmap_k^{-1}(\nbh) \subset\Hsp.$$
If $\BFld(\orig)=0$, then there exists a unique $\ZZZ$-invariant vector field $\AFld$ on $\anbh$ vanishing on $\dHsp$, and such that the following diagram is commutative:
\begin{equation}\label{equ:BFld-lift}
\begin{CD}
& & T\anbh @>{T\Pmap_k}>> T\nbh \\
& & @A{\AFld}AA @AA{\BFld}A \\
\Hsp \;\; & \;\; \supset \;\; & \anbh @>{\Pmap_k}>> \nbh  & \;\; \subset \;\; & \;\; \RRR^2
\end{CD}
\end{equation}
Notice that in general $\AFld$ is smooth only on $\anbh\cap\IHsp$ and is just \emph{continuous\/} on $\Hsp$.

Let $\AFlow_t$ and $\BFlow_t$ be the local flows generated by $\AFld$ and $\BFld$ respectively.
Then for every $t\in\RRR$ the following diagram is commutative 
\begin{equation}\label{equ:BFlow-lift}
\begin{CD}
\anbh @>{\AFlow_t}>> \Hsp \\
@V{\Pmap_k}VV @VV{\Pmap_k}V \\
\nbh @>{\BFlow_t}>> \RRR^2
\end{CD}
\qquad 
\text{i.e.}
\qquad
\Pmap_k\circ\AFlow_t(x) = \BFlow_t\circ\Pmap_k(x),
\end{equation}
provided both parts of this equality are defined.

The following lemma is crucial for the proof of Proposition~\ref{pr:inf-id-shift-func}.
\begin{lem}
If $a,a' \in \anbh$ belong to the same orbit of $\AFlow$, then $b=\Pmap_k(a)$ and $b'=\Pmap_k(a')$ belong to the same orbit of $\BFlow$, see Figure~\ref{fig:time-shift}.
Moreover, the time between $a$ and $a'$ with respect to $\AFlow$ is equal to the time between $b$ and $b'$ with respect to $\BFlow$.
\end{lem}
\begin{proof}
Indeed, if $a'=\AFlow_{\tau}(a)$, then 
$$
b'=\Pmap_k(a')=
\Pmap_k\circ\AFlow_{\tau}(a) =
\BFlow_{\tau}\circ\Pmap_k(a)=
\BFlow_{\tau}(b).
$$
Lemma is proved.
\end{proof}

\begin{figure}[ht]
\includegraphics[height=3cm]{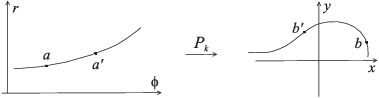}
\caption{}\protect\label{fig:time-shift}
\end{figure}

\begin{lem}\label{lm:formulas-for-F}
Let $\BFunc:\RRR^2\to\RRR$ be a homogeneous polynomial of degree $\dg+1\geq2$, $\HFld=(-\BFunc'_{y},\BFunc'_{x})$ be the Hamiltonian vector field of $\BFunc$, and $$\eta:\RRR^2\to\RRR\setminus\{0\}$$ a smooth everywhere non-zero function.
Consider the following vector field $$\BFld=\eta\HFld=\eta\cdot(-\BFunc'_{y},\BFunc'_{x})$$ and let $\AFld=(\AFld_1,\AFld_2)$ be the vector field on $\Hsp$ induced by $\BFld$ via mapping 
$$\Pmap_1=\Pmap:\Hsp\to\RRR^2, \qquad \Pmap(\phi,\rrho) = (\rrho\cos\phi,\rrho\sin\phi).$$
Write $\BFunc$ in the following form
$$\BFunc(x,y) = y^{a} \Rmap(x,y),$$ 
where $a\geq0$ and $\Rmap$ is a polynomial that is not divided by $y$.
Then 
\begin{equation}\label{equ:F1}
\AFld_1(\phi,\rrho) = 
\frac{(\dg+1)\,\cdot\,\BFunc(\Pmap(\phi,\rrho))}{\rrho^2} = 
\rrho^{\dg-1}\, \phi^{a}\, \gamma_1(\phi),
\end{equation}
for a certain smooth function $\gamma_1:\RRR\to\RRR$ such that $\gamma_1(0)\not=0$.

Moreover, if $a\geq1$, then 
\begin{equation}\label{equ:F2}
\AFld_2(\phi,\rrho) = \rrho^{\dg}\, \phi^{a-1}\,  \gamma_2(\phi),
\end{equation}
where $\gamma_2:\RRR\to\RRR$ is a smooth function such that $\gamma_2(0)\not=0$.
\end{lem}
\begin{cor}\label{cor:formulas-for-F}
If $\BFunc$ has property \AST, then $a=0$ or $1$.
Hence 
\begin{align*}
& \AFld_1(\phi,\rrho) = \rrho^{\dg-1}\gamma_1(\phi), \quad \text{if} \ \ a=0, \\
& \AFld_2(\phi,\rrho) = \rrho^{\dg} \gamma_2(\phi),  \quad \text{if} \ \ a=1. 
\end{align*}
Thus in both cases one of the coordinate functions of $\AFld$ does not divides by $\phi$.
\end{cor}

\begin{proof}[Proof of Lemma~\ref{lm:formulas-for-F}.]
First notice that for a homogeneous polynomial $\BFunc$ of degree $\dg+1$ the following \emph{Euler identity\/} holds true:
\begin{equation}\label{equ:EulerIdent}
 x \BFunc'_{x} + y \BFunc'_{y} = (\dg+1)\BFunc.
\end{equation}
Also, it follows from Lemma~\ref{lm:pres-homog-pol} that every multiple $y$ in $\BFunc$ yields the multiple $\phi$ in $\BFunc\circ\Pmap$. 
Therefore
\begin{equation}\label{equ:QP-rphg}
\BFunc\circ\Pmap(\phi,\rrho) = \rrho^{\dg+1}\, \phi^{a}\, \gamma_1(\phi),
\end{equation}
for a certain smooth function $\gamma_1:\RRR\to\RRR$ such that $\gamma_1(0)\not=0$.

Consider now the Jacobi matrix of $\Pmap$:
$$
J(\Pmap) = 
\left(
\begin{matrix}
-\rrho\sin\phi & \cos\phi \\
\rrho\cos\phi & \sin\phi
\end{matrix}
\right)
$$
Then it follows from the commutative diagram~\eqref{equ:BFld-lift} that 
$$\BFld\circ\Pmap=J(\Pmap)\cdot \AFld,$$ i.e.
$$
\left(
\begin{matrix}
\BFld_1\circ\Pmap \\ \BFld_2\circ\Pmap
\end{matrix}
\right)
= 
\left(
\begin{matrix}
-\rrho\sin\phi & \cos\phi \\
\rrho\cos\phi & \sin\phi
\end{matrix}
\right)
\cdot
\left(
\begin{matrix}
\AFld_1 \\ \AFld_2
\end{matrix}
\right),
$$
whence 
$$
\left(
\begin{matrix}
\AFld_1 \\ \AFld_2
\end{matrix}
\right)
= 
\left(
\begin{matrix}
-\frac{1}{\rrho}\sin\phi & \frac{1}{\rrho}\cos\phi \\
\cos\phi & \sin\phi
\end{matrix}
\right)
\cdot
\left(
\begin{matrix}
\BFld_1\circ\Pmap \\ \BFld_2\circ\Pmap
\end{matrix}
\right).
$$
Therefore
$$
\AFld_1 = \frac{-(\BFld_1\circ\Pmap)\cdot  \sin\phi + (\BFld_2\circ\Pmap)\cdot \cos\phi}{\rrho}.
$$
Denote 
$$
\Amap(x,y) = \frac{-y \BFld_1  + x \BFld_2}{x^2+y^2} = 
\frac{y \BFunc'_{y}  + x \BFunc'_{x}}{x^2+y^2} \cdot \eta \overset{\eqref{equ:EulerIdent}}{=\!=\!=\!=} 
\frac{(\dg+1)\BFunc}{x^2+y^2} \cdot \eta.
$$
Then 
$$
\AFld_1 = \Amap\circ\Pmap
\overset{\eqref{equ:QP-rphg}}{=\!=\!=\!=} 
\rrho^{\dg-1}\, \phi^{a}\, \gamma_1(\phi).
$$
Similarly, 
$$
\AFld_2 =(\BFld_1\circ\Pmap)\cdot  \cos\phi + (\BFld_2\circ\Pmap)\cdot \sin\phi.
$$
Put
$$
\Bmap(x,y) = \frac{x \BFld_1  + y \BFld_2}{\sqrt{x^2+y^2}} = 
\frac{- x \BFunc'_{y}  + y \BFunc'_{x}}{\sqrt{x^2+y^2}} \cdot \eta.
$$
Then $\AFld_2 = \Bmap\circ\Pmap$.
Since the numerator of the latter fraction is a homogeneous polynomial of degree $\dg+1$, it follows from Lemma~\ref{lm:pres-homog-pol} that
$$
\AFld_2=\rrho^{\dg}\, \phi^{a_1}\, \gamma_2(\phi),
$$ 
for certain $a_1\geq0$ and a smooth function $\gamma_2:\RRR\to\RRR$ such that $\gamma_2(0)\not=0$.

It remains to prove that if $a\geq1$ then $$a_1 = a-1.$$
Equivalently, we have to show that the numerator:
$$
N = -x\BFunc'_{y}+y\BFunc'_{x} 
$$
of $\Bmap$ is divided by $y^{a-1}$ but not by $y^{a}$.

Notice that 
$$
\BFunc'_{x} = y^{a} \Rmap'_{x}, \qquad
\BFunc'_{y} = a y^{a-1} \Rmap + y^{a} \Rmap'_{y}.
$$
Whence
$$ 
N = -x\BFunc'_{y}+y\BFunc'_{x}=
-a x y^{a-1} \Rmap - x y^{a} \Rmap'_{y} + y^{a+1} \Rmap'_{x}
$$ 
Since $\Rmap$ is not divided by $y$, it follows that $N$ is divided by $y^{a-1}$ but not by $y^{a}$.
\end{proof}

\section{Correspondence between flat functions}\label{sect:corresp-flat-func}
Recall that a smooth function $\afunc:\RRR^n\to\RRR$ is \emph{flat\/} on a subset $K\subset\RRR^n$ provided all partial derivatives of $\afunc$ of all orders vanish at avery point $x\in K$.

Let $\ARtwo$ be the set of smooth functions $\afunc:\RRR^2\to\RRR$ that are flat at $\orig$

Let also $\AZHR$ be the set of all $\ZZZ$-invariant smooth functions $\hat\alpha:\Hsp\to\RRR$ that are flat on $\dHsp$.

\begin{thm}\label{th:corresp-flat-func}
The mapping 
$$\Pmap_k:\Hsp\to\RRR^2, \qquad \Pmap_k(\phi,\rrho)=(\rrho^k\cos\phi,\rho^k\sin\phi)$$
yields a bijection 
$$
\fbij_k:\ARtwo \to \AZHR, \qquad \fbij_k(\afunc) = \afunc\circ\Pmap_k
$$ 
which is $C^{r,r}_{W,W}$-continuous and its inverse $\fbij_k^{-1}$ is $C^{(2k+1)r,r}_{W,W}$-continuous for every $r\geq 0$.
\end{thm}
\proof
For each $r=0,\ldots,\infty$ let $\Func^{r}(\RRR^2,\orig)$ be the space of all $C^{r}$-functions $\afunc:\RRR^2\to\RRR$ such that $\afunc(\orig)=0$, 
and $\Func^{r}(\Hsp,\dHsp)$ be the space of $\ZZZ$-invariant $C^{r}$-functions $\hat\afunc:\Hsp\to\RRR$ such that $\hat\afunc(\dHsp)=0$.

Then for every $\afunc\in\Func^{0}(\RRR^2,\orig)$ the function $\hat\afunc = \afunc\circ\Pmap_k$ is also continuous on $\Hsp$, $\ZZZ$-invariant, and vanish on $\dHsp$, i.e. $\hat\afunc\in\Func^{0}_{\ZZZ}(\Hsp,\dHsp)$.
Thus we obtain a well-defined mapping 
\begin{equation}\label{equ:q-map-on-cont}
\fbij_k:\Func^{0}(\RRR^2,\orig) \to \Func^{0}_{\ZZZ}(\Hsp,\dHsp), \qquad \fbij_k(\afunc)=\afunc\circ\Pmap_k.
\end{equation}

Conversely, every $\hat\afunc\in\Func^{0}_{\ZZZ}(\Hsp,\dHsp)$ yields a unique function $\afunc\in\Func^{0}(\RRR^2,\orig)$, whence $\fbij_k$ is a bijection.

Since $\Pmap_k$ is smooth, it follows that 
$$\fbij_k\bigl(\,\Func^{\infty}(\RRR^2,\orig)\,\bigr) \; \subset \; \Func^{\infty}_{\ZZZ}(\Hsp,\dHsp)$$  
and the restriction map
$$\fbij_k:\Func^{\infty}(\RRR^2,\orig) \to \Func^{\infty}_{\ZZZ}(\Hsp,\dHsp)$$
is $C^{r,r}_{W,W}$-continuous for every $r=0,\ldots,\infty$.
But this mapping is not onto, e.g. the second coordinate $\rrho:\Hsp\to\RRR$ being a smooth function is the image of the function $(x^2+y^2)^{1/2k}$ which is not differentiable at $\orig\in\RRR^2$.

Suppose that $\afunc$ is flat at $\orig$.
Then it is easy to see that $\hat\afunc$ is flat at every point of $\dHsp$, i.e. \ $\fbij_k(\ARtwo) \subset \AZHR$.
The following Lemma~\ref{lm:estim-for-Dg} shows that in fact
$$\fbij_k(\ARtwo) = \AZHR$$
and the inverse map $\fbij_k^{-1}$ is $C^{(2k+1)r,r}_{W,W}$-continuous for every $r\geq 0$.

\begin{lem}\label{lm:estim-for-Dg}
Suppose that $\hat\afunc\in\AZHR$. 
Let $\afunc=\fbij_k^{-1}(\hat\afunc)$, and 
\begin{equation}\label{equ:dg-ab}
D\afunc = \frac{\partial^{a+b}\afunc}{\partial x^{a}\; \partial y^{b}}
\end{equation}
be a partial derivative of $\afunc$ of order $a+b$.

{\rm(i)}~Then $D\afunc$ is a sum of finitely many functions of the form
$$\frac{A \cdot B}{(x^2+y^2)^{s/2k}},$$
where $A:\RRR^2\to\RRR$ is a smooth function which does not depend on $\afunc$ and 
$$
B = \fbij_k^{-1}\left(  \frac{\partial^{j}\hat\afunc}{\partial \phi^{j_1} \,\partial \rrho^{j_2}}    \right),
\qquad
j = j_1+j_2 \ \leq \ a+b,
$$
and $s$ is positive integer such that\ $s/2k \leq a+b$.
The total number of these functions depends only on $a$ and $b$ and does not depend on $\afunc$.

{\rm(ii)}~$D\afunc$ is a continuous function vanishing at $\orig\in\RRR^2$.
Hence $\afunc$ is a smooth function flat at $\orig\in\RRR^2$, i.e. $\fbij_k$ is a bijection between $\ARtwo$ and $\AZHR$.

{\rm(iii)}~For every $r\geq 0$ and a compact $K\subset\RRR^2$ we have the following estimaiton:
\begin{equation}\label{equ:estim-qinv-cont}
 \|\afunc\|^{r}_{K} \leq C \|\hat\afunc\|^{(2k+1)r}_{L},
\end{equation}
where 
\begin{equation}\label{equ:def-L}
 L  \ = \ \Pmap_k^{-1}(K) \ \; \cap \ \; [0,2\pi] \times [0,\infty),
\end{equation}
and $C>0$ does not depend on $\hat\afunc$.
Whence the inverse mapping $\fbij_k^{-1}$ is $C^{(2k+1)r,r}_{W,W}$-continuous.
\end{lem}

Before proving this lemma we establish some formulas.

\subsection{Formulas for $\Pmap_k^{-1}$ and its derivatives.}
Let $(x,y)\in\RRR^2$.
Then $x^2+y^2=\rrho^{2k}$.
For simplicity suppose that $x>0$, hence 
$$
\rrho=(x^2+y^2)^{\frac{1}{2k}}, \qquad 
\phi = \arctan(y/x) + 2\pi n,
$$
for a certain $n\in\ZZZ$.
Therefore
\begin{align*}
\phi'_{x} & = \frac{-y}{x^2+y^2}, &&
\phi'_{y}   = \frac{x}{x^2+y^2}, \\
\rrho'_{x} & = \frac{x}{k\, (x^2+y^2)^{1-\frac{1}{2k}}}, &&
\rrho'_{y} = \frac{y}{k\, (x^2+y^2)^{1-\frac{1}{2k}}}.
\end{align*}

Similarly, for every $a,b\geq 0$ there exist smooth functions 
$$\mu_i,\nu_i:\RRR^2\to\RRR, \qquad (i=1,\ldots,a+b),$$ such that
\begin{equation}\label{equ:deriv-for-Pinv}
\frac{\partial^{a+b}\phi}{\partial x^{a} \partial y^{b}} = 
\sum_{i=1}^{a+b} \frac{\mu_i}{(x^2+y^2)^{a+b}},
\qquad 
\frac{\partial^{a+b}\rrho}{\partial x^{a} \partial y^{b}} = 
\sum_{i=1}^{a+b} \frac{\nu_i}{(x^2+y^2)^{a+b-\frac{1}{2k}}}.
\end{equation}

These formulas do not depend on a particular choice of the expression of $\phi$ through $x$ and $y$.

\begin{proof}[Proof of Lemma~\ref{lm:estim-for-Dg}.]
(i)~First consider the derivative $\afunc'_x$.
Let $z=(x,y)\not=\orig$.
Then in a sufficiently small neighborhood $\anbh_{z}$ of $z$ we can define an inverse map $\Pmap_k^{-1}:\anbh_{z}\to\Hsp$ such that $\afunc=\hat\afunc\circ\Pmap_k^{-1}$.
Therefore
$$
\afunc'_x = (\hat\afunc'_{\phi}\circ \Pmap_k^{-1}) \cdot \phi'_x + (\hat\afunc'_{\rrho}\circ \Pmap_k^{-1}) \cdot \rrho'_x.
$$

Notice that every partial derivative of $\hat\afunc\in\AZHR$ belongs to $\AZHR$ as well, whence by~\eqref{equ:q-map-on-cont} this derivative determines a unique continuous function on $\anbh_{z}$.
Therefore we can write 
$$
\afunc'_x = \fbij_k^{-1}(\hat\afunc'_{\phi}) \cdot \phi'_x + \fbij_k^{-1}(\hat\afunc'_{\rrho}) \cdot \rrho'_x
 =
\frac{-y \cdot \fbij_k^{-1}(\hat\afunc'_{\phi})}{x^2+y^2} + 
\frac{ x \cdot \fbij_k^{-1}(\hat\afunc'_{\rrho})}{k(x^2+y^2)^{1-\frac{1}{2k}}}. 
$$
Thus we have obtained a desired presentation.
The proof for other partial derivatives of $\afunc$ is similar and we left it to the reader.

(ii)~Let us show the continuity of $D\afunc$.
Denote $$D^j\hat\alpha=\frac{\partial^{j}\hat\afunc}{\partial \phi^{j_1} \,\partial \rrho^{j_2}}.$$
Since $D^j\hat\alpha$ is flat on $\dHsp$, it follows that there exists a smooth function $\xi\in\AZHR$ such that $D^j\hat\alpha=\rrho^{s} \xi$.
Therefore 
$$
B = \fbij_k^{-1}(D^j\hat\alpha) = 
\fbij_k^{-1}(\rrho^{s}) \; \fbij_k^{-1}(\xi) = (x^2+y^2)^{s/2k} \; \fbij_k^{-1}(\xi),
$$
whence 
\begin{equation}\label{equ:reduce2}
\frac{AB}{(x^2+y^2)^{s/2k}} = A \; \fbij_k^{-1}(\xi)
\end{equation}
is continuous.
Hence $D\afunc$ is continuous as well.
Notice that $\xi(\phi,0)=0$.
Therefore $\fbij_k^{-1}(\xi_i)(\orig)=0$, whence $D\afunc(\orig) = 0$.

(iii)~Let $\afunc=\fbij_k^{-1}(\hat\afunc)$.
We have to estimate $\|\afunc\|^{r}_{K}$.
Notice that the subset $L\subset\Hsp$ defined by~\eqref{equ:def-L} is compact and $\Pmap(L)=K$.
Therefore  
\begin{equation}\label{equ:estim-on-0-norms}
\|\fbij_k^{-1}(\hat\afunc)\|^{0}_{K} \ = \ 
\|\afunc\|^0_{K} \ = \ \sup_{x\in K} |\afunc(x)| \ = \
\sup_{(\phi,\rrho)\in L} |\hat\afunc(\phi,\rrho)| \ = \ \|\hat\afunc\|^0_{L}.
\end{equation}
By (ii) and~\eqref{equ:reduce2} every partial derivative $D\afunc$ of $\afunc$ of order $r$ can be represented in the form
$$
D\afunc = \sum_i A_i \cdot \fbij_k^{-1}\left(\frac{D^{j_i}\hat\afunc}{\rrho^{s_i}}\right),
$$
where $A_i$ is smooth on all $\RRR^2$, $D^{j_i}\hat\afunc$ is a partial derivative of $\hat\afunc$ of order $j_i\leq r$, and $s_i\leq 2kr$.

Notice that for every $i$ there are constants $C_1,C_2,C_3>0$ that do not depend on $\hat\afunc$ and such that 
\begin{multline}\label{equ:estim-Aqinv}
\left\|\fbij_k^{-1}\left(\frac{D^{j_i}\hat\afunc}{\rrho^{s_i}}\right)\right\|^{0}_{K} 
 \overset{\text{\eqref{equ:estim-on-0-norms}}}{=\!=\!=\!=} 
\left\|\frac{ D^{j_i}\hat\afunc}{\rrho^{s_i}} \right\|^{0}_{L}  
\ \overset{\text{(Lemma~\ref{lm:Z-cont}})}{\leq} \ \\
 \ \leq \ C_1\, \|D^{j_i}\hat\afunc \|^{s_i}_{L} 
\ \overset{\text{(Lemma~\ref{lm:D-cont})}}{\leq} \
 C_2\, \|\hat\afunc \|^{s_i+j_i}_{L} \ 
\ \overset{\text{\eqref{equ:deriv-for-Pinv}}}{\leq} \
\ C_3 \, \|\hat\afunc \|^{(2k+1)r}_{L}.
\end{multline}

Hence there exists $C_4>0$ such that 
$$
\|D\afunc\|^{0}_{K} \ \leq  \ 
\sum_i \left\| A_i \cdot \fbij_k^{-1}\left(\frac{D^{j_i}\hat\afunc}{\rrho^{k_i}}\right) \right\|^{0}_{K} \ \leq \
C_4\; \|\hat\afunc \|^{(2k+1)r}_{L}.
$$

Therefore
$\|\afunc\|^{r}_{K} \leq C \|\hat\afunc\|^{(2k+1)r}_{L}$
for a certain $C>0$ that depends on $K$ and $r$ but $\hat\afunc$.
\end{proof}
Theorem~\ref{th:corresp-flat-func} is completed.

\section{Correspondence between smooth mappings that are $\infty$-close to the identity}\label{sect:corresp-flat-maps}
Let $\EZHH$ be the set of all smooth maps $$\hdifM=(\hdifM_1,\hdifM_2):\Hsp\to\Hsp,$$
satisfying the following conditions:

\begin{enumerate}
\item[(i)]
 $\hdifM$ is $\ZZZ$-equivariant, i.e. 
\begin{equation}\label{equ:h_Zequiv}
\hdifM_1(\phi+2\pi,\rrho)=\hdifM_1(\phi,\rrho) + 2\pi, \qquad
\hdifM_2(\phi+2\pi,\rrho)=\hdifM_2(\phi,\rrho).
\end{equation}

\item[(ii)] 
$\hdifM$ is fixed on $\dHsp$ and $\hdifM(\IHsp) \subset \IHsp$;

\item[(iii)]
$\difM$ is $\infty$-close to $\id_{\Hsp}$ on $\dHsp$, i.e.\! the following functions
$$
\hdifM_1(\phi,\rrho)-\phi, \qquad
\hdifM_2(\phi,\rrho)-\rrho
$$
are flat on $\dHsp$.
\end{enumerate}

Let also $\ERtwo$ be the set of smooth mappings $\difM:\RRR^2\to\RRR^2$ such that $\difM^{-1}(\orig)=\orig$ and $\difM$ is $\infty$-close to $\id_{\RRR^2}$ at $\orig$.

\begin{lem}\label{lm:hhZ__abZ}
Let $\hdifM=(\hdifM_1,\hdifM_2):\Hsp\to\Hsp$ be a mapping and 
$$
\hat\afunc(\phi,\rrho)=\hdifM_1(\phi,\rrho)-\phi, \qquad  \hat\bfunc(\phi,\rrho)=\hdifM_2(\phi,\rrho)-\rrho.
$$
Then $\hdifM$ is $\ZZZ$-equivariant if and only if the functions $\hat\afunc$ and $\hat\bfunc$ are $\ZZZ$-invariant.
\end{lem}
\begin{proof}
Notice that 
\begin{align*}
\hat\afunc(\phi+2\pi, \rrho) - \hat\afunc(\phi, \rrho) & = 
\hdifM_1(\phi+2\pi,\rrho)-\phi-2\pi - (\hdifM_1(\phi,\rrho)-\phi)  \\ & =
\hdifM_1(\phi+2\pi,\rrho)- \hdifM_1(\phi,\rrho)- 2\pi, \\
\hat\bfunc(\phi+2\pi, \rrho)-\hat\bfunc(\phi, \rrho) & =
\hdifM_2(\phi+2\pi,\rrho)-\rrho - (\hdifM_2(\phi,\rrho)-\rrho)  \\ & =
\hdifM_2(\phi+2\pi,\rrho)- \hdifM_2(\phi,\rrho).
\end{align*}
These identities together with~\eqref{equ:h_Zequiv} imply our statement.
\end{proof}

\begin{thm}\label{th:corresp-flat-maps}
The mapping $\Pmap_k$ yields a $C^{r,r}_{W,W}$-continuous bijection 
$$\mbij_k:\ERtwo\to\EZHH$$
such that its inverse $\mbij^{-1}_k$ is $C^{(2k+1)r,r}_{W,W}$-continuous for every $r\geq0$.
\end{thm}
\proof
Let $\EZHdH$ be the set of all continuous, $\ZZZ$-equivariant mappings $\hdifM:\Hsp\to\Hsp$ that are fixed on $\dHsp$ and $\hdifM(\IHsp)\subset\IHsp$.

Let also $\ERtwoz$ be the set of all continuous maps $\difM:\RRR^2\to\RRR^2$ such that $\difM^{-1}(\orig)=\orig$.

Then every $\hdifM\in\EZHdH$ yields a unique $\difM\in\ERtwoz$ such that the following diagram is commutative:
$$
\begin{CD}
\Hsp @>{\hdifM}>> \Hsp \\
@V{\Pmap_k}VV @VV{\Pmap_k}V \\
\RRR^2 @>{\difM}>> \RRR^2
\end{CD}
$$
i.e. $\difM\circ\Pmap_k=\Pmap_k\circ\hdifM$.
In the coordinate form this means that 
\begin{equation}\label{equ:hP_Phh}
\begin{array}{rcccl}
\difM_1(\rrho^k\cos\phi,\rrho^k\sin\phi) & = & \hdifM_2(\phi,\rrho)^k \cdot \cos\hdifM_1(\phi,\rrho) 
\\
\difM_2(\rrho^k\cos\phi,\rrho^k\sin\phi) & = & \hdifM_2(\phi,\rrho)^k \cdot \sin\hdifM_1(\phi,\rrho). 
\end{array}
\end{equation} 
For such a pair $\difM$ and $\hdifM$ we will use the following notations:
\begin{align}
\hat\afunc(\phi,\rrho) =\hdifM_1(\phi,\rrho) - \phi, \qquad
\hat\bfunc(\phi,\rrho) =\hdifM_2(\phi,\rrho) - \rrho, \label{equ:ha_hb} \\
\gamma(x,y) =\difM_1(x,y) - x, \qquad
\delta(x,y) =\difM_2(x,y) - y. \label{equ:g_d}
\end{align}

Thus the correspondence $\hdifM\mapsto\difM$ is a well-defined mapping 
$$\mbij'_k:\EZHdH\to\ERtwoz.$$

Our aim is to prove that $\mbij'_k$ yields a bijection $$\mbij_{k}^{-1}:\EZHH\to\ERtwo.$$

First let us show that the image of $\mbij'_k$ includes $\ERtwo$.
Indeed, let $\difM\in\ERtwo$.
Since $\difM$ is $C^1$ (actually $C^{\infty}$) and $1$-close to the identity at $\orig$ (actually $\infty$-close), we have that the tangent map
$$T_\orig\difM:T_\orig\RRR^2\to T_\orig\RRR^2$$ is the identity.
Therefore $\difM$ induces a unique mapping $\hdifM:\Hsp\to\Hsp$ fixed on $\dHsp$.
Moreover, since $\difM^{-1}(\orig)=\orig$, we obtain that $\hdifM(\IHsp)=\IHsp$, whence  $\hdifM\in\EZHdH$, and thus $\mbij'_k(\hdifM)=\difM$.

Also notice that a uniqueness of such $\hdifM$ implies that we have a well-defined map 
$$
\mbij_k:\ERtwo\to\EZHdH
$$
inverse to $\mbij'_k$.

It remains to prove the following lemma:

\begin{lem}
$\mbij_k(\ERtwo) = \EZHH$.
Moreover, for every $r\geq0$ the restriction map 
$$\mbij_k: \ERtwo \to \EZHH$$ is $C^{r,r}_{W,W}$-continuous, while its inverse
$$\mbij_k^{-1}:\EZHH\to\ERtwo$$ is $C^{(2k+1)r,r}_{W,W}$-continuous.
\end{lem}
\begin{proof}
Let $\difM\in\ERtwo$ and $\hdifM=\mbij_k(\difM)$.
It suffices to prove that $\hdifM$ is smooth and $\infty$-close to $\id_{\Hsp}$ on $\dHsp$ in a neighborhood of $(0,0)\in\Hsp$.

Since $\difM(\orig)=\orig$ and $\difM$ is $\infty$-close to $\id_{\RRR^2}$ at $\orig$ we have that 
\begin{equation}\label{equ:h1h2}
\difM_1(x,y) = x + x a_1 + y b_1,
\qquad
\difM_2(x,y) = y + x a_2 + y b_2,
\end{equation}
where $a_1,a_2,b_1,b_2\in\ARtwo$.

Then it follows from~\eqref{equ:hP_Phh} and~\eqref{equ:h1h2} that
$$
(\difM_1\circ\Pmap_k)^2 + (\difM_2\circ\Pmap_k)^2 =  \hdifM_2^2 = \rrho^{2k} \cdot (1 + \omega(\phi,\rrho)),
$$
$$
2 \cdot (\difM_1\circ\Pmap_k) \cdot (\difM_2\circ\Pmap_k) =
\hdifM_2^{2k} \cdot \sin 2\hdifM_1=\rrho^{2k} \cdot(\sin2\phi+ \xi(\phi,\rrho))
$$
where $\omega,\xi:\Hsp\to\RRR$ are smooth functions flat on $\dHsp$.
Hence
$$
\sin 2\hdifM_1 = \frac{\sin2\phi+\xi}{1+\omega}=
(\sin2\phi+\xi)(1-\omega+\omega^2-\cdots) = 
\sin2\phi+ \psi,
$$ 
where $\psi$ is smooth in a neighborhood of $(0,0)\in\Hsp$ and flat on $\dHsp$.
Therefore by~\eqref{equ:Hadamard-ffinvxy}
$$
\hdifM_1 = \frac{1}{2}\; \arcsin(\sin2\phi+ \psi) 
\overset{\text{\eqref{equ:Hadamard-ffinvxy}}}{=\!=\!=\!=} \phi + \psi\cdot \tau(\phi,\rrho),
$$
where $\tau$ is smooth in a neighborhood of $(0,0)\in\Hsp$.
Hence $\hdifM_1(\phi,\rrho)-\phi$ is smooth in a neighborhood of $(0,0)\in\Hsp$ and flat on $\dHsp$.

It remains to prove a smoothness of $\hdifM_2$ at every point $(\phi_0,0)$.
Let $A=\cos\phi_0$, $B=\sin\phi_0$.
Then it follows from~\eqref{equ:hP_Phh} and~\eqref{equ:h1h2} that 
\begin{multline*}
A\cdot\difM_1\circ\Pmap_k + B\cdot \difM_1\circ\Pmap_k \  \overset{\eqref{equ:hP_Phh}}{=\!=\!=\!} \
\hdifM_2^k\cdot(A\cos\hdifM_1+B\sin\hdifM_1)\   = \\ = \  \hdifM_2^k \cos(\hdifM_1-\phi_0).
\end{multline*} 
\begin{multline*}
A\cdot\difM_1\circ\Pmap_k + B\cdot \difM_1\circ\Pmap_k \  \overset{\eqref{equ:h1h2}}{=\!=\!=\!}  \
\rrho^k(A\cos\phi+B\sin\phi + c)\ =   \\ =  \ \rrho^k( \cos(\phi-\phi_0) + c),
\end{multline*} 

where $c\in\AZHR$.
Hence 
\begin{equation}\label{equ:hdifM_2}
\hdifM_2(\phi,\rrho)  = \rrho \cdot \underbrace{ \sqrt[k]{\frac{\cos(\phi-\phi_0) + c}{\cos(\hdifM_1-\phi_0)}} }_{\eta} = \rrho \cdot \eta(\phi,\rrho)
\end{equation}
Since $\hdifM_1$ is smooth and $\hdifM_1-\phi$ is flat on $\dHsp$, it follows that in a neighborhood of $(\phi_0,0)$ the function $\eta$ is smooth and $\eta-1$ is flat.
Hence 
$$ 
\hdifM_2 = \rrho + \hat\bfunc,
$$
where $\bfunc\in\AZHR$.
It also follows that $\mbij_k$ is $C^{r,r}_{W,W}$-continuous.

\medskip 
Consider now the map $\mbij_k^{-1}$.
Let $\hdifM=(\hdifM_1,\hdifM_2)\in\EZHH$ and 
$$\difM=\mbij_k^{-1}(\hdifM)=(\difM_1,\difM_2) \ \in \ \ERtwoz.$$

By assumption $\hat\afunc$ and $\hat\bfunc$ are flat on $\dHsp$ and by Lemma~\ref{lm:hhZ__abZ} they are $\ZZZ$-invariant, whence $\hat\afunc,\hat\bfunc\in\AZHR$.
We have to show that $\gamma$ and $\delta$ are smooth and flat at $\orig\in\RRR^2$.
Due to Theorem~\ref{th:corresp-flat-func} it suffices to establish that $\gamma\circ\Pmap_k$ and $\delta\circ\Pmap_k$ belong to $\AZHR$.

By~\eqref{equ:Hadamard-fxy} there are smooth functions $\mu,\nu:\Hsp\to\RRR$ such that 
\begin{align*}
& 
\cos\hdifM_1 = \cos(\phi + \hat\afunc) =
\cos\phi + \hat\afunc \cdot \mu(\phi,\hat\afunc), \\
& 
\sin\hdifM_1 = \sin(\phi + \hat\afunc) =
\sin\phi + \hat\afunc \cdot \nu(\phi,\hat\afunc).
\end{align*}
Evidently, $\mu$ and $\nu$ are $\ZZZ$-invariant.
Also notice that $$\hdifM_2^k = (\rrho+\hat\bfunc)^{k} = \rrho^{k} + \hat\bfunc_1,$$ for some $\hat\bfunc_1\in\AZHR$.
Hence
\begin{equation}\label{equ:gP_dP}
\begin{array}{rl}
 \gamma\circ\Pmap_k(\phi,\rrho) & = 
(\rrho^k+\hat\bfunc_1) (\cos\phi + \hat\afunc \cdot \mu(\phi,\hat\afunc))- \rrho^k\cos\phi =  \\ & =
\hat\bfunc_1 \cdot \cos\phi + (\rrho^k+\hat\bfunc_1)\cdot \hat\afunc \cdot \mu(\phi,\hat\afunc), \\
\delta\circ\Pmap_k(\phi,\rrho) & = 
\hat\bfunc_1 \cdot \sin\phi + (\rrho^k+\hat\bfunc_1)\cdot \hat\afunc \cdot \nu(\phi,\hat\afunc).
\end{array}
\end{equation}

Since $\hat\afunc,\hat\bfunc\in\AZHR$, we see that $\gamma\circ\Pmap_k, \delta\circ\Pmap_k \in\AZHR$ as well.

It remains to note that the mapping $\mbij_k^{-1}$ coincides with the following sequence of correspondences: 
$$
\hdifM
  \;\overset{\eqref{equ:ha_hb}}{\longmapsto}\; 
(\hat\afunc,\hat\bfunc)
  \;\overset{\eqref{equ:gP_dP}}{\longmapsto}\; 
(\gamma\circ \Pmap,\delta\circ\Pmap) 
  \;\overset{\fbij_k}{\mapsto}\;
(\gamma,\delta) 
  \;\overset{\eqref{equ:ha_hb}}{\longmapsto}\; 
\difM,
$$
in which for every $r\geq0$ the first and second arrows are $C^{r,r}_{W,W}$-conti\-nuous and by Theorem~\ref{th:corresp-flat-func} the third one is $C^{(2k+1)r,r}_{W,W}$-conti\-nuous.
Hence $\mbij_k^{-1}$ is $C^{(2k+1)r,r}_{W,W}$-continuous for every $r\geq 0$.
\end{proof}

Theorem~\ref{th:corresp-flat-maps} is completed.

\section{Proof of Proposition~\ref{pr:inf-id-shift-func}.}\label{sect:proof_of_shift2}
Let $\BFld$ be a smooth vector field, defined in a neighborhood $\nbh$ of the origin $\orig\in\RRR^2$.
Suppose that $\BFld$ has property \AST\ at $\orig$.
Therefore we can assume that $\BFld=\eta\HFld$, where $\eta:\RRR^2\to\RRR\setminus\{0\}$ is everywhere non-zero smooth function and $\HFld=(-\BFunc'_{y},\BFunc'_{x})$ is a Hamiltonian vector field of a certain homogeneous polynomial $\BFunc:\RRR^2\to\RRR$ of degree $\dg+1\geq2$ having no multiple factors.

Denote by $\BFlow$ the corresponding local flow of $\BFld$.

For every $\difM\in\EndBVi{\infty}$ we have to find a smooth function $$\afunc:\nbh\to\RRR$$ which is flat at $\orig$ and such that 
$$
\difM(z) = \BFlow(z,\afunc(z)).
$$

Let $\Pmap:\Hsp\to\RRR^2$ be the map defining polar coordinates, i.e.\! $$\Pmap(\phi,\rrho)=(\rrho\cos\phi,\rrho\sin\phi).$$
Thus $\Pmap=\Pmap_1$ in the notation of Section~\ref{sect:polar-coordinates}.

Set $\anbh=\Pmap^{-1}(\nbh)$.

Let $\Flat(\nbh,\orig)$ be the space of smooth functions $\nbh\to\RRR$ which are flat at $\orig$, and $\Flat_{\ZZZ}(\anbh,\dHsp)$ be the space of smooth $\ZZZ$-invariant functions $\anbh\to\RRR$ which are flat on $\dHsp$.

Denote by $\FlatMap(\nbh,\RRR^2,\orig)$ the space of smooth maps $\difM:\nbh\to\RRR^2$ such that $\difM^{-1}(\orig)=\orig$ and $\difM$ is $\infty$-close to $\id_{\nbh}$ at $\orig$.
Finally, let $\FlatMap_{\ZZZ}(\anbh,\Hsp,\dHsp)$ be the space of smooth $\ZZZ$-equivariant mappings $\hdifM:\anbh\to\Hsp$ such that $\hdifM^{-1}(\dHsp)=\dHsp$ and $\hdifM$ is $\infty$-close to $\id_{\anbh}$ at every points of $\dHsp$.

Then it follows from Theorems~\ref{th:corresp-flat-func} and~\ref{th:corresp-flat-maps} that the mapping $\Pmap$ yields the following bijections $\fbij_1$ and $\mbij_1$ which for simplicity we denote by $\fbij$ and $\mbij$ respectively: 
$$
\fbij: \Flat(\nbh,\orig) \to \Flat_{\ZZZ}(\anbh,\dHsp),
$$
$$
\mbij: \FlatMap(\nbh,\RRR^2,\orig) \to \FlatMap_{\ZZZ}(\anbh,\Hsp,\dHsp).
$$

Let $\AFld$ be the lifting of the vector field $\BFld$ from $\nbh$ to $\anbh$ via $\Pmap$.
Denote by $\EndAUDHi{\infty}$ the subset of $\EndAU$ consisting of mappings that are $\infty$-close to $\id_{\Hsp}$ on $\dHsp$.
Moreover, let $\EndZAUDHi{\infty}$ be the subset of $\EndAUDHi{\infty}$ consisting of $\ZZZ$-equivariant maps.
Then we have the following inclusions:
$$ 
\begin{CD}
\FlatMap(\nbh,\RRR^2,\orig)\quad & \supset & \quad \EndBVi{\infty}   \\
  @V{\mbij}VV \\
\FlatMap_{\ZZZ}(\anbh,\Hsp,\dHsp) \quad & \supset & \quad \EndZAUDHi{\infty}.
\end{CD}
$$
\begin{lem}\label{lm:reduct-G-to-F}
$\mbij\bigl( \; \EndBVi{\infty} \; \bigr) = \EndZAUDHi{\infty}$.
\end{lem}
\begin{proof}
Let 
$$\difM\in\EndBVi{\infty}\qquad\text{³}\qquad
\hdifM=\mbij(\difM)\in\FlatMap_{\ZZZ}(\anbh,\Hsp,\dHsp).$$
We have to show that $\hdifM\in\EndZAUDHi{\infty}$, i.e. 
\begin{enumerate}
\item[(i)]
$\hdifM$ is a diffeomorphism in a neighborhood of every singular point point $z\in\Sigma_{\AFld} = \dHsp$ of $\AFld$;
\item[(ii)]
$\hdifM(\hat\omega\cap\anbh)\subset\hat\omega$ for every orbit $\hat\omega$ of $\AFld$.
\end{enumerate}

{\bf  Proof of (i).}
Since $\difM$ is $\infty$-close to $\id_{\RRR^2}$ at $\orig$, it follows from Theorem~\ref{th:corresp-flat-maps} that $\hdifM$ is $\infty$-close to the identity on $\Sigma_{\AFld} = \dHsp$.
Therefore for every point $z\in\dHsp$ the corresponding tangent map $T_{z}\hdifM:T_{z}\Hsp\to T_{z}\Hsp$ is identity and thus it is nondegenerate.

{\bf  Proof of (ii).}
Let $\hat\omega$ be an orbit of $\AFld$ and $\omega=\Pmap(\hat\omega)$ be the corresponding orbit of $\BFld$.
Then by definition $\difM(\omega\cap\nbh)\subset\omega$.
Hence $\hdifM(\hat\omega\cap\anbh)$ is included in some orbit $\hat\omega_1$ of $\AFld$ which is also mapped onto $\omega$ by $\Pmap$, i.e.\! $\Pmap(\hat\omega_1)=\omega$.

We have to show that $\hat\omega=\hat\omega_1$.
Actually this follow from the structure of orbits of $\BFld$.

Indeed, suppose that $\BFunc$ is a product of definite quadratic forms, i.e. $\BFunc(z)\not=0$ for $z\not=0$.
Then the structure of the orbits of $\AFld$ and $\BFld$ for this case is shown in Figure~\ref{fig:def}.
It follows from this figure that $\hat\omega=\Pmap^{-1}(\omega)$, whence $\hat\omega=\hat\omega_1$.

Suppose that $\BFunc$ has linear factors.
Then, see~Figure~\ref{fig:non-def}, the set $\BFunc^{-1}(0)$ is a union of $2l$ rays $T_0,\ldots,T_{2l-1}$ for $i=1,\ldots,l$ starting at the origin $\orig$ and such that $T_i$ and $T_{i+l\!\! \mod 2l}$ belong to the same straight line.
Moreover, the set $\Pmap^{-1}\circ\BFunc^{-1}(\orig)$ is a union of $\dHsp$ together with countable set of vertical half-lines $\hat T_j$, $(j\in\ZZZ)$.
We can assume that $\Pmap(\hat T_{j}) = T_{j \mod 2l}$.

Since $\difM(T_i)=T_i$, it follows that $\hdifM(\hat T_j)$ for all $i$ and $j$.
Therefore $\Pmap$ yields a bijection between the orbits of $\BFld$ laying in the angles between $T_{i}$ and $T_{i+1}$ and orbits of $\AFld$ laying between $\hat T_{i+2ls}$ and $\hat T_{i+1+2ls}$, $(s\in\ZZZ)$.
Hence $\hat\omega=\hat\omega_1$.

Thus $\mbij\bigl( \EndBVi{\infty} \bigr) \subset \EndZAUDHi{\infty}$.

Conversely, let $\hdifM\in\EndZAUDHi{\infty}$ and $\difM=\mbij^{-1}(\hdifM)\in\FlatMap(\nbh,\RRR^2,\orig)$.
We have to show that $\difM \in \EndBVi{\infty}$.
Since $\difM$ is $\infty$-close to $\id_{\RRR^2}$ at $\orig$, we obtain that $\difM$ is a local diffeomorphism at every (actually unique) singular point of $\BFld$.
Moreover, let $\omega$ be any orbit of $\BFld$ and $\hat\omega$ be an orbit of $\AFld$ such that $\omega=\Pmap(\hat\omega)$.
Then by definition $\hdifM(\hat\omega\cap\anbh) \subset \hat\omega$.

Since $\Pmap\circ\hdifM = \difM \circ \Pmap$, we obtain that 
$$
\difM(\omega\cap\nbh) \ \ \subset \ \ \difM \circ \Pmap(\hat\omega \cap \anbh) \ \ = \ \ \Pmap\circ\hdifM (\hat\omega \cap \anbh) \ \ \subset \ \ \Pmap(\hat\omega) \ \ = \ \
\omega. 
$$

Thus $\EndZAUDHi{\infty} \subset \mbij\bigl( \EndBVi{\infty} \bigr)$.
\end{proof}

It remains to prove the following statement:

\begin{prop}\label{prop:lift-shift-map}
Suppose that $\BFunc$ has property \AST.
Then there exists a unique mapping
$$
\iShift: \EndZAUDHi{\infty} \to \Flat_{\ZZZ}(\anbh,\dHsp)
$$
such that 
$$
\hdifM(x) = \AFlow(x,\iShift(\hdifM)(x))
$$
for all $\hdifM\in\EndZAUDHi{\infty}$.
This map is $C^{r+\dg,r}_{W,W}$-continuous.
\end{prop}

\begin{cor}
Define the mapping $\ShiftI:\EndBVi{\infty} \to \Flat(\nbh,\orig)$ by
$\ShiftI = \fbij^{-1} \circ \iShift \circ \mbij$, i.e.\! so that the following diagram becomes commutative: 
$$
\begin{CD}
\FlatMap_{\ZZZ}(\anbh,\Hsp,\dHsp)\; & \;\supset\; & \;\EndZAUDHi{\infty} @>{\iShift}>> \Flat_{\ZZZ}(\anbh,\dHsp) \\
@A{\mbij}AA @A{\mbij}AA    @AA{\fbij}A \\
\FlatMap(\nbh,\RRR^2,\orig)\; & \;\supset\; & \;\EndBVi{\infty} @>{\ShiftI}>> \Flat(\nbh,\orig) 
\end{CD}
$$
Then $\ShiftI$ satisfies the statement of Proposition~\ref{pr:inf-id-shift-func}.
\end{cor}
\begin{proof}[Proof of Corollary.]
Indeed, let $\difM\in\EndBVi{\infty}$,
$$
\hdifM=\mbij(\difM)\in\EndZAUDHi{\infty},
\qquad
\hat\afunc=\iShift(\hdifM)\in\Flat_{\ZZZ}(\anbh,\dHsp).
$$
So
$$
\hdifM(a)=\AFlow(a,\hat\afunc(a)), \qquad \forall a\in\anbh.
$$
Set
$$
\afunc\;=\;\fbij^{-1}(\hat\afunc)\;=\;\fbij^{-1}\circ \secfunc \circ \mbij(\difM)\;\in\; \Flat(\nbh,\orig),
$$
thus $\hat\afunc=\afunc\circ\Pmap$.
First we have to show that 
$$
\difM(b) = \BFlow(b,\afunc(b)), \qquad \forall b\in\nbh.
$$
Let $a\in\anbh$ and $b\in\nbh$ be such that $b=\Pmap(a)$.
Then
\begin{multline*}
\difM(b) = 
\difM\circ\Pmap(a) = 
\Pmap\circ\hdifM(a) = 
\Pmap\circ\AFlow(a,\hat\afunc(a))= \\ =
\BFlow(\Pmap(a),\hat\afunc(a))=
\BFlow(\Pmap(a),\afunc\circ\Pmap(a))=
\BFlow(b,\afunc(b)).
\end{multline*}

It remains to prove continuity of $\ShiftI$.

Notice that for every $r\geq \dg$ the mapping $\mbij$ is $C^{r,r}_{W,W}$-continuous, $\iShift$ is $C^{r,r-\dg}_{W,W}$-continuous, and 
$\fbij^{-1}$ is $C^{r-\dg, [(r-\dg)/3]}_{W,W}$-continuous, where $[t]$ is the integer part of $t\in\RRR$.
Hence $\ShiftI$ is $C^{r, [(r-\dg)/3]}_{W,W}$-continuous of all $r\geq\dg$.

Replacing $r$ by $3r+\dg$ we obtain that $\ShiftI$ is $C^{3r+\dg, r}_{W,W}$-continuous.
\end{proof}

Thus Proposition~\ref{pr:inf-id-shift-func} and therefore Theorem~\ref{th:global-shift-func-1} are proved modulo Proposition~\ref{prop:lift-shift-map}.

\begin{rem}\label{rem:division-by-r}\rm
Let $A\in\Flat(\anbh,\dHsp)$, i.e.\! $A$ is flat on $\dHsp$.
Then it follows from the Hadamard lemma that for every $t\in\NNN$ there exists $A_t\in\Flat(\anbh,\dHsp)$ such that $A=\rrho^{t} A_t$.
\end{rem}

\subsection*{Proof of Proposition~\ref{prop:lift-shift-map}.}
Let $\hdifM=(\hdifM_1,\hdifM_2)\in\EndAUDHi{\infty}$.
Since all orbits of $\AFld$ in $\IHsp$ are non-closed, it follows that for every $z\in\IHsp$ there exists a unique number $\secfunc(z)\in\RRR$ such that 
$$
\hdifM(z) = \BFlow(z,\secfunc(z)).
$$
Thus we get a shift-function $\secfunc:\IHsp\to\RRR$ for $\hdifM$.
Moreover, it follows from~\eqref{equ:reg-shifts} that this function is smooth on $\IHsp$.

Define $\secfunc$ on $\dHsp$ by $\secfunc(z)=0$ for $z\in\dHsp$.
We have show that this extension is smooth of $\Hsp$ and flat on $\dHsp$.

Let $\phi_0\in\dHsp$.
Then by Lemma~\ref{lm:pres-homog-pol} 
$$\BFunc\circ\Pmap(\phi,\rrho) =  \rrho^{\dg+1} (\phi-\phi_0)^a \gamma(\phi),$$
for some $a\geq0$ depending on $\phi_0$ and a smooth function $\gamma:\RRR\to\RRR$ such that $\gamma(\phi_0)\not=0$.

Moreover, since $\BFunc$ has property \AST, it follows from Corollary~\ref{cor:formulas-for-F} that $a$ is either $0$ or $1$.
 
Consider two cases.
Not loosing generality, we can also assume that $\phi_0=0$.

1) Suppose that $a=0$, i.e.
$$\BFunc\circ\Pmap(\phi,\rrho) =  \rrho^{\dg+1} \gamma(\phi),$$
is a neighborhood of $(0,0)\in\Hsp$.
Equivalently, this means that $\BFunc$ is not divided by $y$.
Then by~\eqref{equ:F1} of Lemma~\ref{lm:formulas-for-F} we have that 
$$ 
\AFld_1(\phi,\rrho) = \rrho^{\dg-1}\, \gamma_1(\phi).
$$ 
Since $\hdifM_1-\phi$ and $\hdifM_2-\rrho$ are flat on $\dHsp$, they are divided by $\rrho$, whence we can write 
$$
\hdifM_1(\phi,\rrho)=\phi+\Amap(\phi,\rrho),\qquad
\hdifM_2(\phi,\rrho)=\rrho+\rrho\Bmap(\phi,\rrho),
$$
where $\Amap,\Bmap\in\Flat(\anbh,\dHsp)$.

Notice that $\AFld$ defines a the following system of ODE:
$$
\left\{
\begin{array}{l}
\dot\phi = \AFld_1(\phi,\rrho) \\
\dot\rrho = \AFld_2(\phi,\rrho).
\end{array}
\right.
$$
Whence $dt = \frac{d\phi}{\AFld}$.
Therefore the time $\secfunc(\phi,\rrho)$ between the points $(\phi,\rrho)$ and $\hdifM(\phi,\rrho)$ can be calculated by the following formula:
$$
\secfunc(\phi,\rrho) = 
\int\limits_{\phi}^{\hdifM_1(\phi,\rrho)} \frac{d\theta}{\rrho^{\dg-1}\, \gamma(\theta)}.
$$
We will show that $\secfunc$ is smooth in a neighborhood of $(0,0)\in\Hsp$.
It suffices to prove that $\secfunc$ has smooth partial derivatives of the first order which are flat on $\dHsp$.

An easy calculation shows that 
$$
\secfunc'_{\phi}(\phi,\rrho) = 
\frac{(\hdifM_1)'_{\phi}}{\hdifM_2^{\dg-1}\cdot \gamma(\hdifM_1)} -
\frac{1}{\rrho^{\dg-1} \cdot \gamma},
\quad\quad
\secfunc'_{\rrho}(\phi,\rrho) = 
\frac{(\hdifM_1)'_{\rrho}}{\hdifM_2^{\dg-1} \gamma(\hdifM_1)}.
$$
Notice that 
$$
(\hdifM_1)'_{\phi} = 1 +  \Amap'_{\phi},
\qquad
(\hdifM_1)'_{\rrho} = \Amap'_{\rrho}.
$$
Moreover, 
\begin{equation}\label{equ:h_2_p_case1}
\hdifM_2^{\dg-1} = \rrho^{\dg-1} (1 + \bar\Bmap),
\qquad
\gamma(\hdifM_1(\phi,\rrho))=\gamma(\phi) (1 + C),
\end{equation}
for some $\bar\Bmap,C\in\Flat(\anbh,\dHsp)$.
Hence
\begin{multline}\label{equ:d_secfun_d_phi:1}
\secfunc'_{\phi}(\phi,\rrho) = 
\frac{1 +\Amap'_{\phi}}{\rrho^{\dg-1} (1 + \bar\Bmap)\gamma(\hdifM_1)} -
\frac{1+C}{\rrho^{\dg-1}\gamma(\hdifM_1)} = \\ =
\frac{\overbrace{\Amap'_{\phi}- \bar\Bmap-C -  \bar\Bmap C}^{D}}{\rrho^{\dg-1} (1 + \bar\Bmap)\gamma(\hdifM_1)} =
\frac{D/\rrho^{\dg-1}}{(1 + \bar\Bmap)\gamma(\hdifM_1)}.
\end{multline}
Since $D\in\Flat(\anbh,\dHsp)$, it follows from the Hadamard lemma, see  Remark~\ref{rem:division-by-r}, that $D/\rrho^{\dg-1}$ and therefore $\secfunc'_{\phi}(\phi,\rrho)$ belong to $\Flat(\anbh,\dHsp)$.

Similarly,
\begin{equation}\label{equ:d_secfun_d_rho:1}
\secfunc'_{\rrho}(\phi,\rrho) = 
\frac{\Amap'_{\rrho}}{\rrho^{\dg-1} (1+\bar\Bmap) \gamma(\hdifM_1)} =
\frac{\Amap'_{\rrho}/\rrho^{\dg-1}}{ (1+\bar\Bmap) \gamma(\hdifM_1)}.
\end{equation}
Again this function is smooth since $\Amap'_{\rrho}\in\Flat(\anbh,\dHsp)$.

2) Suppose that $a=1$.
Then $\BFunc=y\Rmap$, where $\Rmap(x,0)\not=0$ and by~\eqref{equ:F2} of Lemma~\ref{lm:formulas-for-F}
$$
\AFld_2(\phi,\rrho)  = \rrho^{\dg}\, \gamma_2(\phi).
$$
Since $\AFld_1(0,\rrho)=0$, we see that the half-axis $\{\phi=0,\rrho>0\}$ is the orbit of $\AFld$.
Therefore $\hdifM$ preserves this half-axis, i.e. $\hdifM_1(0,\rrho)=0$, whence by the Hadamard lemma we obtain that 
$$
\hdifM_1(\phi,\rrho)=\phi + \phi \Amap(\phi,\rrho),
\qquad
\hdifM_2(\phi,\rrho)=\rrho + \rrho\Bmap(\phi,\rrho)
$$
for certain $\Amap,\Bmap\in\Flat(\anbh,\dHsp)$.
Therefore
$$
\secfunc(\phi,\rrho) = 
\int\limits_{\rrho}^{\hdifM_2(\phi,\rrho)} \frac{d \rho}{\rho^{\dg}\, \gamma(\phi)}.
$$
Then similarly to the previous case it can be shown that 
\begin{equation}\label{equ:d_secfun_d_phi:2}
\secfunc'_{\phi}(\phi,\rrho) = 
\frac{\Bmap'_{\phi}/\rrho^{\dg}}{(1+\bar\Bmap)\gamma(\hdifM_1)},
\end{equation}
and
\begin{equation}\label{equ:d_secfun_d_rho:2}
\secfunc'_{\rrho}(\phi,\rrho) = 
\frac{E/\rrho^{\dg}}{(1+\bar\Bmap)\gamma(\hdifM_1)},
\end{equation}
where similarly to~\eqref{equ:h_2_p_case1}
$\bar\Bmap, C, E$ are defined by
$$
\hdifM_2^{\dg} = \rrho^{\dg} (1 + \bar\Bmap),
\qquad
\gamma(\hdifM_1(\phi,\rrho))=\gamma(\phi) (1 + C),
$$
$$
E = \Bmap'_{\rrho}-\bar\Bmap- C -\bar\Bmap C
$$
and belong to $\Flat(\anbh,\dHsp)$.
Hence $\secfunc\in\Flat(\anbh,\dHsp)$ as well.

It remains to prove continuity of the correspondence $\hdifM\mapsto\secfunc$.
Notice that the expressions for $\secfunc'_{\phi}$ and $\secfunc'_{\rrho}$ include division by  $\rrho^{\dg}$ and the operators $\partial/\partial\phi$, and $\partial/\partial\rrho$.
Recall that by Lemmas~\ref{lm:D-cont} and~\ref{lm:Z-cont} the division by $\rrho$ and differentiating by $\phi$ and $\rrho$ are $C^{r+1,r}_{W,W}$-continuous.

It follows from formulas~\eqref{equ:d_secfun_d_phi:1}, \eqref{equ:d_secfun_d_rho:1}, \eqref{equ:d_secfun_d_phi:2}, and \eqref{equ:d_secfun_d_rho:2} that there exists $d>0$ and a closed ball $K\subset\nbh$ containing $\orig\in\RRR^2$ such that the absolute values of denominators of these expressions are greater than $2d$ at every point of $K$.
Put $$L = \Pmap^{-1}(K) \; \cap \; [0,2\pi]\times[0,\infty).$$
Then it follows from expressions for $\secfunc'_{\phi}$ and $\secfunc'_{\rrho}$ and Lemmas~\ref{lm:D-cont} and~\ref{lm:Z-cont} that for every $r\geq0$ and $\varepsilon>0$ there exists $\delta\in(0,d)$ such that the inequality
$$
\|\hdifM-\qdifM\|^{r+\dg+1}_{K}<\delta \qquad\text{implies}\qquad
\|\secfunc(\hdifM)-\secfunc(\qdifM)\|^{r+1}_{L}<\varepsilon.
$$
Hence the correspondence $\hdifM\mapsto\secfunc$ is $C^{r+\dg,r}_{W,W}$-continuous for all $r\geq0$.
We leave the details to the reader.

\section{Acknowledgments}
I am sincerely greateful to V.\;V.\;Sharko, E.\;Polu\-lyakh, A.\;Prishlyak, I.\;Vlasenko, and I.\;Yurchuk for useful discussions and interest to this work.

\bigskip
\bigskip
\noindent
{\sc Sergiy Maksymenko} \\
Topology dept., \\ Institute of Mathematics of NAS of Ukraine \\ 
Tereshchenkivs'ka st. 3, Kyiv, 01601 Ukraine \\
email: \texttt{maks@imath.kiev.ua}

\end{document}